\documentclass[final,3p,times]{elsarticle}
    \makeatletter
    \def\ps@pprintTitle{%
       \let\@oddhead\@empty
       \let\@evenhead\@empty
       \let\@oddfoot\@empty
       \let\@evenfoot\@oddfoot
    }
    \makeatother

\usepackage{amssymb}
\usepackage{amsmath}

\usepackage{lineno}
\usepackage{mathtools}
\usepackage{float}

\newtheorem{theorem}{Theorem}
\newtheorem{lemma}[theorem]{Lemma}
\newdefinition{remark}{Remark}
\newproof{proof}{Proof}
\newdefinition{algo}{Algorithm}
\newdefinition{assumption}{Assumption}

\newcommand{\nr}[1]{\left\|{#1}\right\|}
\newcommand{\R}{\mathbb{R}}
\DeclareMathOperator{\argmin}{argmin}
\newcommand{\bigo}{\mathcal{O}}
\newcommand{\goto}{\rightarrow}

\journal{Applied Math Letters}

\begin{document}
\begin{frontmatter}



\title{Fast convergence to higher multiplicity zeros}


\author{Sara Pollock}

\address{Department of Mathematics,
University of Florida,
Gainesville, FL, 32611}

\begin{abstract}
In this paper, the Newton-Anderson method, which results from applying an 
extrapolation technique known as Anderson acceleration to Newton's method,
is shown both analytically and numerically
to provide superlinear convergence to non-simple roots of scalar equations.
The method requires neither a priori knowledge of the multiplicities of the roots, 
nor computation of any additional function evaluations or derivatives.
\end{abstract}

\begin{keyword}
Rootfinding \sep
nonlinear acceleration \sep
non-simple roots \sep
Newton's method \sep 
Anderson acceleration
\MSC[2010]{65B05,65H04}
\end{keyword}

\end{frontmatter}
%

\section{Introduction.}
Solving nonlinear equations is a problem of fundamental importance in numerical 
analysis, and across many areas of science, engineering, finance and mathematics.
In general, solving nonlinear equations is an iterative process, accomplished by
generating a sequence of approximations to the solution. 
One of the most common methods of obtaining a solution to the nonlinear 
problem $f(x) = 0$ is Newton's method, in which the sequence of approximations to
a zero of $f$ is generated, given some initial $x_0$, by
\begin{align}\label{eqn:newton}
x_{k+1} = x_k - [f'(x_k)]^{-1}f(x_k).
\end{align}

The purpose of this manuscript is to introduce for functions
$f: \mathbb R \rightarrow \mathbb R$, the sequence  where given $x_0$, 
$x_1$ is found by \eqref{eqn:newton}, then for $k \ge 1$, $x_{k+1}$ is
generated by
\begin{align}\label{eqn:na-iter}
        x_{k+1}  = x_k - \frac{x_k - x_{k-1}}
        {f(x_k)/f'(x_k)-f(x_{k-1})/f'(x_{k-1})} 
        (f(x_k)/f'(x_k)).
\end{align}
It will be shown that the iterative scheme \eqref{eqn:na-iter} 
gives fast (superlinear) convergence to roots of 
multiplicity $p > 1$, where the Newton method gives only slower linear convergence. 
It will also be shown how this sequence is the result 
of applying an extrapolation method known as Anderson acceleration \cite{anderson65}
to the Newton iteration \eqref{eqn:newton}.

Newton's method is well-known for its quadratic convergence to simple zeros, supposing 
the iteration is started close enough to some root of a function $f$.  
However, some problems may have non-simple (higher multiplicity) roots.
For a root $c$ of multiplicity $p > 1$, Newton's method converges only linearly,
and $ \lim_{k \rightarrow \infty }|x_{k+1} -c|/|x_k - c| = 1-1/p$,
\cite[Section 6.3]{Quarteroni06}.  
A modified Newton method 
\begin{align}\label{eqn:modnewton}
x_{k+1} = x_k - p\frac{f(x_k)}{f'(x_k)},
\end{align}
can be seen to restore quadratic convergence. However, this requires knowledge of
the multiplicity $p$ of the root which is generally {\em a priori} unknown.

Even if $p$ is unknown, it may be approximated in the course of the iterative process.
One method that does just that is introduced in \cite[Section 6.6.2]{Quarteroni06} 
whereby $x_{k+1} = x_k - p_k f(x_k)/f'(x_k)$ 
gives an approximate or adaptive modified Newton method with $p_0 = 1$ and 
for $k \ge 1$, 
\[
p_k = \frac{x_{k-1} - x_{k-2}}{2x_{k-1}- x_k - x_{k-2}},
\]
where $p_k$ is recomputed on each iteration where the convergence rate is 
sufficiently stable (see \cite[Program 56]{Quarteroni06}).
The method \eqref{eqn:na-iter} introduced here uses a different approximation to 
$p_k$, and will be shown to compare favorably to the adaptive method of 
\cite{Quarteroni06} in the numerical tests of Section \ref{sec:numerics}.

Finally it is remarked that another approach to quadratic convergence for non-simple
roots discussed in for instance \cite{mathews89} is a modified Newton-Raphson method
\[
  x_{k+1} = x_k - \frac{f(x_k)f'(x_k)}{(f'(x_k))^2 - f(x_k)f''(x_k)},
\]
which bears close resemblence to Halley's method \cite{ScTh95}.  
However, the computation of the 
second derivative may be considered unnecessarily laborious as it will be seen 
in Section \ref{sec:numerics} that the superlinear convergence of the method 
\eqref{eqn:na-iter} introduced here converges very nearly as fast as the modified 
Newton method \eqref{eqn:modnewton} but without the {\em a priori} knowledge of 
the multiplicity of the zero.

\section{Anderson accelerating Newton's method.}
To understand the derivation of the method \eqref{eqn:na-iter}, the Anderson 
acceleration algorithm for fixed-point iterations is next introduced.  
This method, which uses a history of the $m+1$ most recent iterates and update steps
to define the next iterate, was introduced by D. G. Anderson in 1965 \cite{anderson65}
in the context of integral equations. It has since increased in popularity and 
become known as an effective method for improving the convergence rate of fixed-point
iterations $x_{k+1} = \phi(x_k)$, and is used in many applications in scientific 
computing \cite{EPRX19,K18,WaNi11}.
The basic Anderson acceleration algorithm with depth $m$ (without damping) 
applied to the fixed-point problem $\phi(x) = x$ for 
$\phi: \mathbb R^n \rightarrow \mathbb R^n$, is shown below. To clarify how it
is applied to a Newton iteration, the following notation is introduced. 
The fixed-point iteration may be written as
$x_{k+1} = \phi(x_k) = x_k + (\phi(x_k)-x_k) = x_k + w_{k+1}$, where
$w_{k+1} = \phi(x_k) - x_k$.  Thus, if $\phi(x_k) = x_k - [f'(x_k)]^{-1}f(x_k)$ 
as in the Newton iteration \eqref{eqn:newton}, 
the update step is $w_{k+1} = -[f'(x_k)]^{-1}f(x_k)$, 
(or, $w_{k+1} = -f(x_k)/f'(x_k)$, in the special case of $f: \R \goto \R$).

\begin{algo}
\label{alg:anderson} 
{\bf(Anderson iteration with depth $m$)} \quad Set depth $m \ge 0$. Choose $x_0$.
Compute $w_1$. Set $x_1 = x_0 + w_1$.
\\ \noindent
For $k = 1, 2, \ldots$, set $m_k = \min\{k,m\}$
\\ \indent
Compute $w_{k+1}$
\\ \indent
Set $F_k= \begin{pmatrix}(w_{k+1}-w_k) & \ldots & (w_{k-m+2} - w_{k-m+1})\end{pmatrix}$,
and $E_k= \begin{pmatrix}(x_{k}-x_{k-1}) & \ldots & (x_{k-m+1} - x_{k-m})\end{pmatrix}$\\ \indent
Compute $\gamma_{k} = \argmin_{\gamma \in \R^m} \nr{w_{k+1} - F_k \gamma}$
\\ \indent
Set $x_{k+1} = x_k + w_{k+1} - \left(E_k + F_k \right)\gamma_{k}$
\end{algo}
A discussion on what norm might be used for the optimization in 
Algorithm \ref{alg:anderson},
and how the minimization problem is solved, can be found in \cite{WaNi11}. 
In the present context of finding a zero of $f: \mathbb R \rightarrow \mathbb R$,
one only needs consider depth $m=1$ ($m=0$ is the original fixed-point iteration), 
and the optimization step reduces to solving a linear equation.
\subsection{Scalar Newton-Anderson}\label{subsec:an1}
In the scalar case, $f: \mathbb R \rightarrow \mathbb R$,
the optimization problem in Algorithm \ref{alg:anderson} 
reduces to a linear equation for a single coefficient, 
$w_{k+1} - \gamma_k(w_{k+1}- w_k) = 0$,
solved by
$ \gamma_k = w_{k+1}/(w_{k+1}-w_k).$
Then the accelerated iterate $x_{k+1}$, for $k \ge 1$ is given by
\begin{align}\label{eqn:seclike}
x_{k+1} 
= x_k +  w_{k+1} - \frac{w_{k+1}}{w_{k+1}-w_k} ((x_k + w_{k+1} - (x_{k-1} + w_k)) 
= x_k -\frac{x_k - x_{k-1}}{w_{k+1}-w_k} w_{k+1}.
\end{align}
If the fixed-point scheme $x_{k+1} = x_k + f(x_k)$ is used, then $w_{k+1} = f(x_k)$,
which results in the secant method. 
If the Newton method \eqref{eqn:newton} is accelerated, then plugging in 
$w_{k+1} = -f(x_k)/f'(x_k)$ yields \eqref{eqn:na-iter}.  

It is remarked here that for $f: \mathbb R^n \rightarrow \mathbb R^n$, 
(or a more general normed vector space) the 
Anderson Algorithm \ref{alg:anderson} applied to $w_{k+1} = f(x_k)$ results in 
a method shown to be a type of multi-secant method \cite{eyert96,FaSa09}, as 
compared to the standard secant method in the scalar case.
One of the motivations for looking at the scalar version of Algorithm 
\ref{alg:anderson} applied to the Newton method was to understand the method
in this simpler setting to gain insight into its use in a more general setting.
As a result of this investigation, and as demonstrated in \cite{PoSw19}, 
it was found for 
$f:\R^n \rightarrow \R^n$,
the Newton-Anderson method can provide superlinear convergence to solutions of
degenerate problems, those whose Jacobians are singular at a solution (and for which
Newton converges only linearly), as well
as nondegenerate problems (where Newton converges quadratically).
This paper focuses on  $f: \R \rightarrow \R$, 
and provides analytical and numerical results to characterize the scalar case.
Next, the Newton-Anderson rootfinding method is summarized,
then its convergence properties are analyzed in the section that follows.

\begin{algo}
\label{alg:na1}
{\bf(Newton-Anderson rootfinding method)}\\
Choose $x_0$.
Compute $w_1 = -f(x_0)/f'(x_0)$. Set $x_1 = x_0 + w_1$
\\ \noindent
For $k = 1, 2, \ldots$
\\ \indent
Compute $w_{k+1}= -f(x_k)/f'(x_k)$
\[
\text{Set} \quad x_{k+1} = x_k -\frac{x_k - x_{k-1}}{w_{k+1}-w_k}w_{k+1} \hspace{4.25in}
\] 
\end{algo}

\section{Rootfinding.}
To give further insight into the main result, a trivial case is first considered.
The Newton method \eqref{eqn:newton} finds the zero of $f(x) = ax-b$, 
or in monic form $f(x) = x-c$, exactly in one step. 
Similarly, it is easily seen that Algorithm \ref{alg:na1} locates the zero of 
$f(x) = (x-c)^p$ with $p> 0$ and $p \ne 1$ after the first optimization step: that is,
$x_2 = c$, if the operations are performed in exact arithmetic.
The modified Newton method \eqref{eqn:modnewton} has a similar property: assuming
$p$ is known, then given $x_0$, the first iterate $x_1 = x_0 - p(x_0-c)/p = c$.
However, one might also show that for 
$f: \mathbb R^n \rightarrow \mathbb R^n$ defined by 
an $n \times n$ invertible matrix $A$,
and $b \in \mathbb R^n$,
the system of $n$ equations given by $f_i(x) = ((Ax - b)_i)^{p_i}$ for exponents
$p_i > 0,$ is solved after the first full optimization step of Algorithm 
\eqref{alg:anderson} with depth $m$ applied to iteration \eqref{eqn:newton} if
there are exactly $m$ distinct exponents $p_i$, $i = 1, \ldots, n$
(a numerical demonstration of this is shown in \cite{PoSw19}).

Next, a more general scalar problem is considered. Suppose $c$ is a non-simple root
of a function $f(x)$ expressed in the form $f(x) = (x-c)^p g(x)$, $ p > 1$,
for some function $g$
which is assumed not to have a zero (or pole) in some neighborhood of $c$.  
The following lemma shows 
the Newton-Anderson rootfinding method approximates the modified Newton method
\eqref{eqn:modnewton}; and, it makes a precise statement regarding how
$(x_{k+1}-x_k)/w_{k+1}= (x_k-x_{k-1})/(w_{k+1}-w_k),$ 
provides an approximation to the multiplicity of the zero of $f$ at $x=c$.
The theorem that follows provides a local convergence analysis of 
Algorithm \ref{alg:na1}.

An alternative approach to the analysis might be to exploit the interpretation
of \eqref{eqn:seclike} as a secant method used to find the (simple) zero of
$-f(x)/f'(x)$, yielding the usual order of convergence for the secant method,
 $(1 + \sqrt 5)/2$. The results that follow, however, give a direct proof that
the method has an order of convergence of at least $(1+ \sqrt 5)/2$; and, 
show that it gives an accurate approximation to the multiplicity of the root (also 
demonstrated numerically in Section \ref{sec:numerics}), which can be of use if
deflation is used to find additional roots.
To fix some notation for the remainder of this section, let $e(x) = c-x$,
and let $\mathcal I_k = (\min\{x_k, x_{k-1}\}, \max\{x_k,x_{k-1}\})$.

\begin{lemma}\label{lem:papprox}
Let $f(x) = (x-c)^p g(x)$ for $p > 1$ where $g: \mathbb R \rightarrow \mathbb R$ is 
a $C^2$ function for which both $g'(x)/g(x)$ and $g''(x)/g(x)$ are bounded
in an open interval $\mathcal N$ containing $c$.  

Define the constants
\begin{align}\label{eqn:defMg}
M_0 
= \max_{x \in \mathcal N} \frac 1 p \left| \left(\frac{g'(x)}{g(x)}\right)^2 
- \frac{g''(x)}{g(x)}\right|,
~\text{ and }~
M_1 = \max_{x \in \mathcal N} \frac 1 p \left| \frac{g'(x)}{g(x)} \right|.
\end{align}

Then, if $x_{k-1},x_k \in \mathcal N_0 \coloneqq 
\{ x \in \mathcal N: e(x)^2 < M_0^{-1} ~\text{ and }~ |e(x)| < M_1^{-1}\}$, 
the iterate $x_{k+1}$ given
by Algorithm \ref{alg:na1} satisfies
$x_{k+1} = x_k + p_k w_{k+1}$ with
\begin{align}\label{eqn:papprox}
p_k = \left(p - 2 e(\eta_k) \frac{g'(\eta_k)}{g(\eta_k)} 
  + \mathcal O(e(\eta_k)^2)\right),
~\text{ for some }~ \eta_k \in \mathcal I_k.
\end{align}
\end{lemma}
The hypotheses on $g$ maintain that $g$ is reasonably smooth
and does not have a zero in the vicinity of $c$.

\begin{proof}
The Newton update step is $w_k = -f(x_{k-1})/f'(x_{k-1})$ so writing
$w_k = w(x_{k-1})$, the update step 
from Algorithm \ref{alg:na1} reads as
\begin{align}\label{eqn:na1g-001}
x_{k+1} = x_k - \left( \frac{x_k - x_{k-1}}{ w(x_k) - w(x_{k-1})}\right) w_{k+1}
= x_k - p_k w_{k+1},
\end{align}
with $p_k = (x_k - x_{k-1})/(w(x_k) - w(x_{k-1})).$
The aim is now to show $p_k \rightarrow p$ as $e(x_k) \rightarrow 0$.

For $f(x)$ given by $f(x) = (x-c)^p g(x)$, the first two derivatives are given
by
\begin{align}\label{eqn:fderiv}
f'(x) &= (x-c)^{p-1}\big(p g(x) + (x-c) g'(x)\big)
\nonumber \\ 
f''(x) &= (x-c)^{p-2}\big(p(p-1) g(x) + 2 p (x-c)g'(x) + (x-c)^2 g''(x)\big).
\end{align}
Writing $w(x_k)$ in terms of $f(x_k) = (x-c)^p g(x_k)$, and $f'(x_k)$ given by 
\eqref{eqn:fderiv} gives
$w(x_{k}) = e(x_k)g(x_k)/(pg(x_k) - e(x_k)g'(x_k)),$
whose denominator is bounded away from zero for $x_k \in \mathcal N_0$.  
Then by the mean value theorem, there is an $\eta_k \in \mathcal \mathcal I_k$
for which $w(x_k) - w(x_{k-1}) = w'(\eta_k)(x_k - x_{k-1})$, by which 
\eqref{eqn:na1g-001} reduces to 
$x_{k+1} = x_k -  w_{k+1}/w'(\eta_k)$.
Temporarily dropping the subscript on $\eta_k$ for clarity of notation, 
taking the derivative of  $w(\eta) = -f(\eta)/f'(\eta)$ yields
\[
\frac{-1}{w'(\eta)} = \frac{f'(\eta)^2}{ (f'(\eta))^2 - f(\eta)f''(\eta)}.
\]
Applying the expansions of $f'$ and $f''$ from \eqref{eqn:fderiv},
cancelling common factors of $e(x)^{p-2}$ and simplifying allows
\begin{align}\label{eqn:na1g-002}
\frac{-1}{w'(\eta)} 
&= \frac{\big(p g(\eta) - e(\eta)g'(\eta)\big)^2}
{ \big(p g(\eta) - e(\eta)g'(\eta)\big)^2 - 
 g(\eta)\big(p(p-1)g(\eta) - 2pe(\eta)g'(\eta) + e(\eta)^2 g''(\eta)\big) } 
\nonumber \\
& = \frac{p - 2 e(\eta) \frac{g'(\eta)}{g(\eta)} 
+ \frac 1 p e(\eta)^2 \left(\frac{g'(\eta)}{g(\eta)} \right)^2}
{1 + \frac 1 p e(\eta)^2 \left(\left(\frac{g'(\eta)}{g(\eta)}\right)^2 -
\frac{g''(\eta)}{g(\eta)} \right)}.
\end{align}
By hypothesis, $x_{k}$ and $x_{k-1}$ are in $\mathcal N_0$
which implies $\eta_k \in \mathcal I_k \subset \mathcal N_0$, so the 
denominator of the right hand side of \eqref{eqn:na1g-002} is of the form 
$1 + \alpha$ with 
$|\alpha| < 1$. Expanding the denominator in a geometric series shows that
\begin{align}
-\frac{1}{w'(\eta_k)} & = p - 2e(\eta_k)\frac{g'(\eta_k)}{g(\eta_k)} 
+ \mathcal O(e(\eta_k)^2).
\end{align}
This shows there is an $\eta_k \in \mathcal I_k$ for which the update
\eqref{eqn:na1g-001} of Algorithm \ref{alg:na1} satisfies
\eqref{eqn:papprox}.
\qed
\end{proof}
\begin{remark}
The adaptive method of \cite[(6.39)-(6.40)]{Quarteroni06} and the current method
both take the form $x_{k+1} = x_k + p_k w_{k+1}$, so it makes sense to compare
the two expressions for $p_k$.  
Letting $\{y_k\}$ represent the sequence generated by 
\cite[(6.39)-(6.40)]{Quarteroni06}, and setting 
$w_k = -f(y_{k-1})/f'(y_{k-1}),$ the resulting iteration may be written
\begin{align*}
y_{k+1} = y_k + \frac{y_{k-1}-y_{k-2}}{ (y_k - y_{k-1}) - (y_{k-1} - y_{k-2})}w_{k+1}
=y_k + \frac{y_{k-1}-y_{k-2}}{ p_{k-1}w_k - p_{k-2}w_{k-1}} w_{k+1}, 
\end{align*}
which differs from update \eqref{eqn:na1g-001} of Algorithm \ref{alg:na1}
both in terms of the set of iterates used in the numerator of $p_k$: 
$\{y_{k-1},y_{k-2}\}$ compared to $\{x_k,x_{k-1}\}$; and, in the 
form of the denominator $p_{k-1}w_k - p_{k-2}w_{k-1}$ as opposed to 
$w_{k+1}-w_k$.  As such, $p_k$ of the adaptive scheme appears more complicated
to analyze as an approximation to $p$, and the two methods will only be compared 
numerically, in the two examples of Section \ref{sec:numerics}.          
\end{remark}
The previous Lemma \ref{lem:papprox} shows the update step of
Algorithm \ref{alg:na1} is of the form $x_{k+1} = x_k + p_k w_k$ where
$p_k \rightarrow p$ so long as $x_k \rightarrow c$.  The next theorem shows that 
$x_k \rightarrow c$, and that the order of convergence is greater than one
(and, in fact, no worse than $(1 + \sqrt 5)/2$).
\begin{theorem}\label{thm:na1con}
Let $f(x) = (x-c)^p g(x)$, for $p > 1$ where $g: \mathbb R \rightarrow \mathbb R$ is 
a $C^2$ function for which both $g'(x)/g(x)$ and $g''(x)/g(x)$ are bounded
in an open interval $\mathcal N$ containing $c$.  
Define the interval
$\mathcal N_1 = \{ x \in \mathcal N_0  :  |e(x)| <  1/(2M_1)\}$,
where $\mathcal N_0$ and $M_1$ are given in the statement of Lemma \ref{lem:papprox}.
Then there exists an interval $\mathcal N_\ast \subseteq \mathcal N_1$ such that
if $x_{k-1},x_k \in \mathcal N_\ast$, all subsequent iterates remain in
$\mathcal N_\ast$ and the iterates defined by
Algorithm \ref{alg:na1} converge superlinearly to the root $c$.
\end{theorem}

\begin{proof}
Suppose $x_k, x_{k-1} \in \mathcal N_1$.
Let $p_k = (x_k - x_{k-1})/(w_{k+1}- w_k)$. Then the error in iterate $x_{k+1}$
satisfies
\begin{align}\label{eqn:tc-001}
e(x_{k+1}) = c -(x_k - p_k w_{k+1}) = e(x_k) + p_kw_{k+1}.
\end{align}
Similarly to the computations of the previous lemma
\[
w_{k+1} = -\frac{f(x_k)}{f'(x_k)} = \frac{e(x_k) g(x_k)}{p g(x_k) - e(x_k)g'(x_k)},
\]
which together with \eqref{eqn:tc-001} shows
\begin{align}\label{eqn:tc-002}
e(x_{k+1}) = e(x_k)\left(1 - \frac{p_k g(x_k)}{pg(x_k) - e(x_k) g'(x_k)} \right)
 = e(x_k)\left(1 - \frac{p_k/p  }{1  - \frac{1}{p}e(x_k) \frac{g'(x_k)}{g(x_k)}} \right).
\end{align}
For $x_k \in \mathcal N_0$ the denominator of \eqref{eqn:tc-002} can be
expanded as a geometric series to obtain
\begin{align}\label{eqn:tc-003}
e(x_{k+1}) 
& = e(x_k)\left(1 - \frac{p_k}{p}
\sum_{j = 0}^{\infty }\left( \frac{1}{p}e(x_k) \frac{g'(x_k)}{g(x_k)} \right)^j
\right).
\end{align}
For $x_k$ and $x_{k-1}$ in $\mathcal N_1 \subset \mathcal N_0$, the results of Lemma 
\ref{lem:papprox} hold, and applying the resulting expansion of $p_k$ 
to \eqref{eqn:tc-003} shows
\begin{align}\label{eqn:tc-004}
e(x_{k+1}) & = 
e(x_k) \left\{ 1- \left(1 - \frac{2}{p} e(\eta_k) \frac{g'(\eta_k)}{g(\eta_k)} 
+ \mathcal O(e(\eta_k)^2)\right) 
\sum_{j = 0}^{\infty }\left( \frac{1}{p}e(x_k) \frac{g'(x_k)}{g(x_k)} \right)^j 
 \right\}
\nonumber \\& = 
e(x_k) \left\{ 1- \left(1 - \frac{2}{p} e(\eta_k) \frac{g'(\eta_k)}{g(\eta_k)} 
+ \mathcal O(e(\eta_k)^2)\right) 
\left( 1 +  \frac{1}{p}e(x_k) \frac{g'(x_k)}{g(x_k)} + 
\left( \frac{1}{p}e(x_k) \frac{g'(x_k)}{g(x_k)} \right)^2
\sum_{j = 0}^{\infty }\left( \frac{1}{p}e(x_k) \frac{g'(x_k)}{g(x_k)} \right)^j 
 \right) \right\}
\nonumber \\& = 
e(x_k) \left\{ 1- \left(1 - \frac{2}{p} e(\eta_k) \frac{g'(\eta_k)}{g(\eta_k)} 
+ \mathcal O(e(\eta_k)^2)\right) 
\left( 1 +  \frac{1}{p}e(x_k) \frac{g'(x_k)}{g(x_k)} + \mathcal{O}(e(x_k)^2) 
 \right) \right\},
\end{align}
for some $\eta_k \in \mathcal I_k$.
Multiplying out terms in \eqref{eqn:tc-004} shows the error satisfies 
\begin{align}\label{eqn:tc-005}
e(x_{k+1})   
 = e(x_k)e(\eta_k) \frac 2 p\frac{g'(\eta_k)}{g(\eta_k)} 
-  e(x_k)^2 \frac1 p \frac{g'(x_k)}{g(x_k)} + \mathcal O (e(x_k) e(\eta_k)^2),
\end{align}
which, for $x_k, x_{k-1}$ in an interval $\mathcal N_\ast \subseteq \mathcal N_1$,
shows the iterates stay in $\mathcal N_\ast$, and converge superlinearly to 
$c$.
\qed
\end{proof}
The standard secant method, when used to approximate a simple root,  
has an order of convergence of  $(1 + \sqrt 5)/2$, and 
the lowest order term in its error expansion is multiple of $e(x_k) e(x_{k-1})$. 
From \eqref{eqn:tc-005}, the lowest order term in the error expansion of 
Newton-Anderson, when approaching a higher-multiplicity root, is a multiple of
$e(x_k) e(\eta_k)$, where $\eta_k$ (from a mean value theorem) 
is between $x_k$ and $x_{k-1}$. This implies
the order of convergence for the method is at least $(1+\sqrt 5)/2$, and generally
less than 2 unless $g'/g \goto 0$. 

\section{Numerical examples}\label{sec:numerics}
In this section, some numerical examples are given to illustrate the efficiency of 
the Newton-Anderson rootfinding method.
In these examples, the proposed method, Algorithm \ref{alg:na1}, 
is compared with the Newton method \eqref{eqn:newton},
the modified Newton method \eqref{eqn:modnewton} (assuming {\em a priori} knowledge
of the multiplicity $p$ of the zero), and the adaptive method of 
\cite[Section 6.6.2]{Quarteroni06}, implemented as described therein.
Additionally, results are shown for the secant method (using $x_0$ as stated, and 
$x_{-1} = x_0 - 10^{-3}$), and the predictor-corrector (PC) 
Newton method of \cite{McWo14}.
The secant method is included because, as shown in \eqref{eqn:seclike},
scalar Newton-Anderson can be interpreted as a secant method applied to the Newton
update step, or a secant method to find the zero of $w(x) = -f(x)/f'(x)$. 
The predictor-corrector method (which was designed to accelerate
Newton's method for simple roots only) comparatively demonstrates the robustness
of Newton-Anderson, which performs comparably in each case tested. In contrast, the
predictor-corrector method outperforms the Newton and secant methods in the first 
example, and is outperformed by both of them in the second.

\subsubsection{Example 1}\label{subsubsec:exQ}
The first example is taken from \cite[Example 6.11]{Quarteroni06}.
The problem tested is finding the zero of $f(x) = (x^2-1)^q \log x$, which has
a zero of multiplicity $p = q+1$ at $x=1$. 
The condition to exit the iterations are those from 
\cite[Example 6.11]{Quarteroni06}, namely 
$|x_{k+1}- x_k| < 10^{-10}$. The iteration counts starting from  $x_0 = 0.8$ 
(for standard, adaptive and modifed Newton methods) 
agree with those stated in \cite{Quarteroni06}. 
Tables \ref{tab:Q6.11-2}-\ref{tab:Q6.11-6} show the respective iteration counts for 
$q = \{2, 6\}$ for each method starting from initial iterates $x_0 = \{0.8, 2, 10\}$.
The final value of $p_k$ is shown in parentheses after the iteration count for the 
Newton-Anderson and adaptive Newton methods. 

Consistent with the analysis from 
Lemma \ref{lem:papprox} and Theorem \ref{thm:na1con}, the performance of the
Newton-Anderson method is linked to its accurate approximation of the root's 
multiplicity. For the result below in Tables \ref{tab:Q6.11-2}-\ref{tab:Q6.11-6},
the final value of $p_k$ in Newton-Anderson was accurate to $\bigo(10^{-8})$, 
except for the last experiment in Table \ref{tab:Q6.11-2}, 
where it was $\bigo(10^{-7})$. 
\begin{table}[ht!]
\centering
{\renewcommand{\arraystretch}{1.1}
\begin{tabular}{|r||c|c|c||c|c|c|}
\hline
$x_0$ & modified N. & N. Anderson & adaptive N. & Newton & P.C. Newton & secant
 \\ [2pt]
\hline
 0.8 & 4 & 6 (3.0000) & 13 (2.9860) & 51 & 38 & 72\\
 2.0 & 5 & 7 (3.0000) & 17 (3.0178) & 56 & 40 & 79\\
10.0 & 7 & 8 (3.0000) & 30 (4.1984) & 63 & 46 & 89\\ 
\hline
\end{tabular}
\caption{Iterations to $|x_{k+1}- x_k| < 10^{-10}$ for $f(x) = (x^2-1)^2\log x$}
\label{tab:Q6.11-2}
}
\end{table}
\begin{table}[ht!]
\centering
{\renewcommand{\arraystretch}{1.1}
\begin{tabular}{|r||c|c|c||c|c|c|}
\hline
$x_0$ & modified N. & N. Anderson & adaptive N. & Newton & P.C. Newton & secant
 \\ [2pt]
\hline\hline
 0.8 & 5 &  7 (7.0000) & 18 (6.7792)  & 127  & 106 & 179 \\
 2.0 & 6 &  8 (7.0000) & 29 (7.3274)  & 140  & 110 & 198 \\
10.0 & 8 & 10 (7.0000) & 80 (12.1095) & 162  & 114 & 229 \\ 
\hline
\end{tabular}
\caption{Iterations to $|x_{k+1}- x_k| < 10^{-10}$ for $f(x) = (x^2-1)^6\log x$}
\label{tab:Q6.11-6}
}
\end{table}

\subsubsection{Example 2}\label{subsubsec:exexp}
The second example concerns finding the zero of 
$f(x) = (x-2)^6\exp(-(x-2)^2/2)$, which has
a zero of multiplicity 6 at $x=2$.
The first 30 differences between consecutive iterates $|x_{k+1}-x_k|$ are shown 
below in Figure \ref{fig:exptest6} starting each iteration from the initial
$x_0 = \{0, 1\}$. 

\begin{figure}[ht!]
\centering
\includegraphics[trim = 0pt 4pt 10pt 20pt,clip = true, width=0.40\textwidth]
{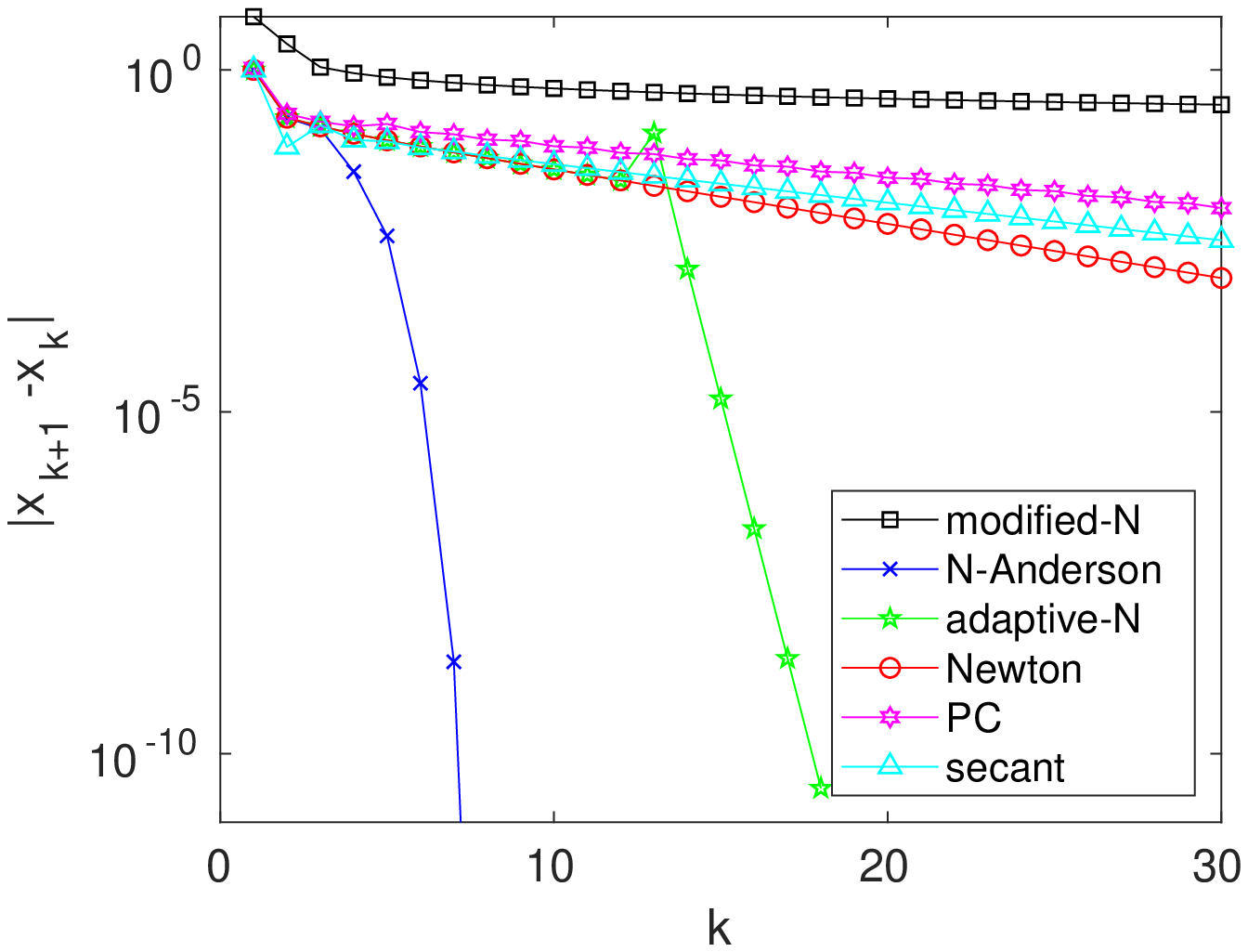}~\hfil~
\includegraphics[trim = 0pt 4pt 10pt 20pt,clip = true, width=0.40\textwidth]
{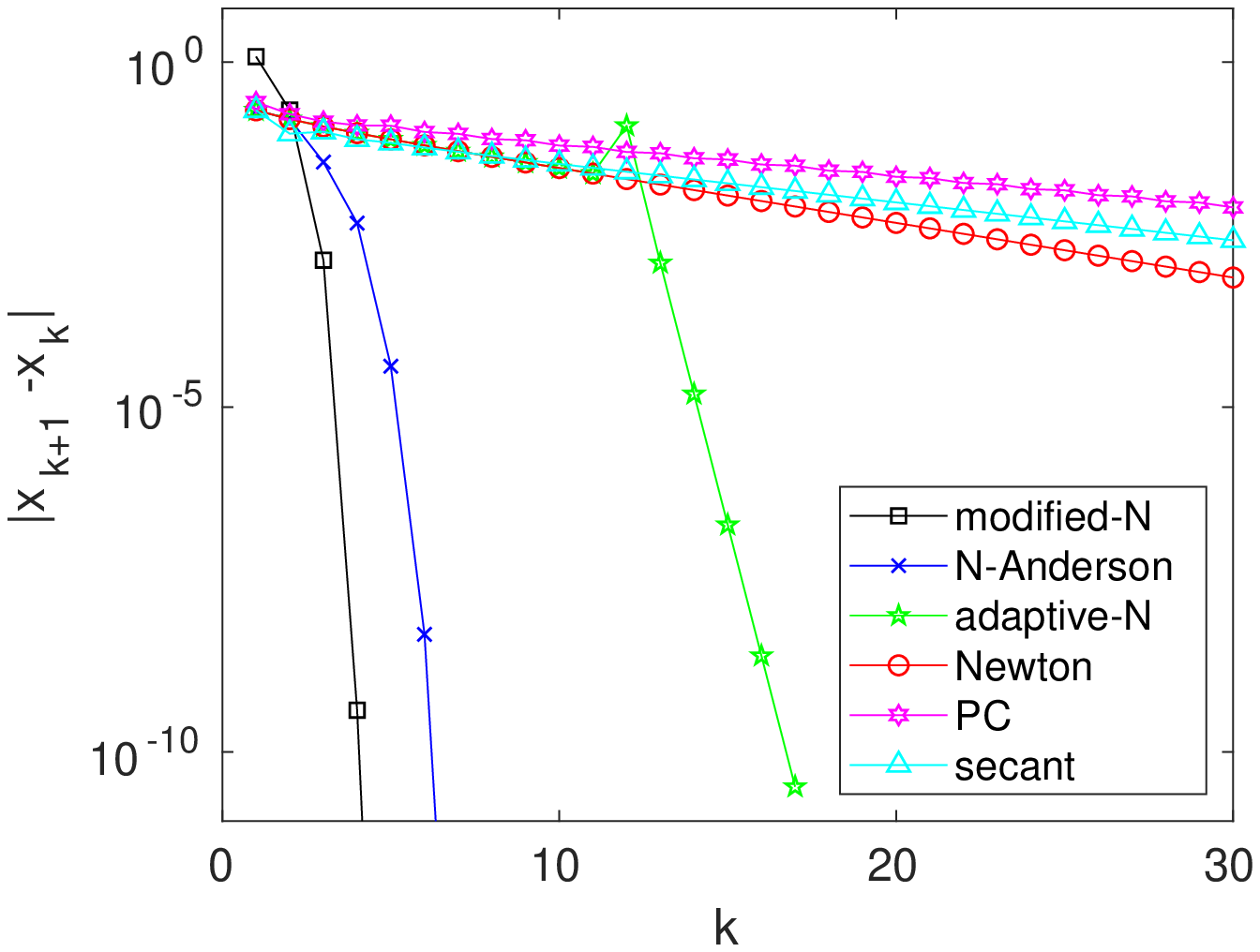}
\caption{$|x_{k+1}- x_k|$ for $f(x) = (x-2)^6(-\exp(x-2)^2/2)$. Left: 
iteration starting with $x_0 = 0$. Right: iteration starting with $x_0 = 1$.
\label{fig:exptest6}
}
\end{figure}

For this problem the modified Newton method
converges in two fewer iterations than Newton-Anderson (12 less than
the adaptive method) starting from $x_0 = 1$, however 
it fails to converge starting from $x_0 = 0$ (it is attracted to the asymptotic zero
as $x \rightarrow \infty$). The Newton-Anderson method has
comparable performance and converges to the same zero in both cases, as do the
remaining methods, although their convergence is substantially slower.

\subsubsection{Numerical order of convergence}
To demonstrate the order of convergence of Newton-Anderson for some instances of
each problem,
the sequence of approximate convergence orders $q_k = \log|x_k -c|/\log|x_{k-1}-c|$
is shown below in Table \ref{tab:qk}. 
\begin{table}[ht!]
\centering
{\renewcommand{\arraystretch}{1.1}
\begin{tabular}{|r||c||c|c|c|c|c|c|c|}
\hline
$f(x)$ & $x_0$ & $q_3$ & $q_4$ & $q_5$ & $q_6$ & $q_7$ & $q_8$ & $q_9$
 \\ [2pt]
\hline\hline
$(x^2-1)^2\log(x)$ & 10 & 2.1027 & 2.4792 & 1.8078 & 1.7729 & 1.6879 &        & \\
$(x^2-1)^4\log(x)$ & 10 & 1.4009 & 3.2797 & 1.9022 & 1.7976 & 1.7032 & 1.6737 & \\
$(x^2-1)^6\log(x)$ & 10 & 0.7489& 5.3648 & 2.0311 & 1.8145 & 1.7161 & 1.6802 &1.6531\\
\hline
$(x-2)^6\exp(-(x-2)^2/2)$  & 0 & 5.6924 & 2.3193 & 2.3055 & 2.0713 &         & &\\
$(x-2)^8\exp(-(x-2)^2/2)$  & 0 & 21.221 & 2.9109 & 2.3625 & 2.1500 &  2.0728 & &\\
$(x-2)^{10}\exp(-(x-2)^2/2)$ & 0 & 12.3187 & 3.0433 & 2.3309 & 2.1592 & 2.0688 & &\\ 
\hline
\end{tabular}
\caption {Approximate orders of convergence $q_k = \log|x_k -c|/\log|x_{k_1}-c|$
          for Examples \ref{subsubsec:exQ} and \ref{subsubsec:exexp}.}
\label{tab:qk} 
}
\end{table}
The convergence orders of the first example behave essentially as predicted, generally
staying in the range $((1 + \sqrt 5)/2,2)$, whereas the approximate convergence orders
from the second example are generally above 2.  This can be easily understood however
as the constant in front of the lowest order term of \eqref{eqn:tc-005} is a multiple
of $g'(\eta_k)/g(\eta_k)$, which for $g(x) =\exp(-(x-2)^2/2)$, goes to 
zero as $x$ approaches $c=2$. 

\section{Conclusion}\label{sec:conclusion}
The purpose of this discussion is to understand the iteration derived from
applying Anderson acceleration to a scalar Newton iteration.
The resulting Newton-Anderson rootfinding method is shown 
to approximate a modified Newton method that yields quadratic convergence 
to non-simple roots.  The presented method does not require {\em a priori}
knowledge of the multiplicity of the root, nor does it require additional
function evaluations or the computation of additional derivatives.
The convergence analysis of the Newton-Anderson method 
demonstrates a local order of convergence of at least $(1 + \sqrt 5)/2$. 
The numerical examples demonstrate this and show
iteration counts close to that of the modified Newton method, and with 
less sensitivity to the initial guess. In comparison with the adaptive Newton 
method of \cite{Quarteroni06} designed to accomplish the same task, the 
implementation is simpler as additional heuristics are not involved, and on
the examples tested, convergence is faster as a more accurate approximation
of the root's multiplicity is attained. Altogether this makes the 
Newton-Anderson rootfinding method worthy of consideration in 
situations involving non-simple roots. 
\section*{Acknowledgements} The author was partially supported by NSF DMS 1852876.

\bibliographystyle{elsarticle-num}
\bibliography{ander}


\end{document}